\documentclass[11pt]{article}%
\usepackage{amsmath}
\usepackage{amsfonts}
\usepackage{amssymb}
\usepackage{graphicx}%
\providecommand{\U}[1]{\protect\rule{.1in}{.1in}}
\newtheorem{theorem}{Theorem}

\newtheorem{definition}[theorem]{Definition}
\newtheorem{example}[theorem]{Example}

\newtheorem{remark}[theorem]{Remark}

\newenvironment{proof}[1][Proof]{\noindent\textbf{#1.} }{\ \rule{0.5em}{0.5em}}
\begin{document}

\title{A method for computation of scattering amplitudes and Green functions of whole
axis problems}
\author{Ra\'{u}l Castillo-P\'{e}rez$^{a}$, Vladislav V. Kravchenko$^{b}$, Sergii M.
Torba$^{b}$\\$^{a}$ {\small Instituto Polit\'{e}cnico Nacional, Maestr\'{\i}a en
Telecomunicaciones, ESIME Zacatenco, CDMX, Mexico}\\$^{b}$ {\small Departamento de Matem\'{a}ticas, Cinvestav, Unidad
Quer\'{e}taro, Quer\'{e}taro, Mexico;}}
\maketitle

\begin{abstract}
A method for the computation of scattering data and of the Green function for
the one-dimensional Schr\"{o}dinger operator $H:=-\frac{d^{2}}{dx^{2}}+q(x)$
with a decaying potential is presented. It is based on representations for the
Jost solutions in the case of a compactly supported potential obtained in
terms of Neumann series of Bessel functions (NSBF), an approach recently
developed in \cite{KNT 2015}. The representations are used for calculating a
complete orthonormal system of generalized eigenfunctions of the operator $H$
which in turn allow one to compute the scattering amplitudes and the Green function
of the operator $H-\lambda$ with $\lambda\in\mathbb{C}$.

\end{abstract}

\section{Introduction}

The Schr\"{o}dinger one-dimensional operator $H:=-\frac{d^{2}}{dx^{2}}+q(x)$
is considered with a potential $q$ defined on the whole axis and decaying at
infinity. A method for the computation of the egenfunctions and generalized
eigenfunctions as well as of scattering data and the Green function is
developed. The need for the computation of the Green function and scattering
data for this operator arises, e.g., in hydroacoustic problems (see, e.g.,
\cite{Barrera1} and \cite{Barrera2}) as well as when solving nonlinear
evolution equation by means of the inverse scattering transform method (see
\cite{Ablowitz}).

The simplest method for computing the Green function requires two linearly
independent solutions (the Jost solutions) of the Schr\"{o}dinger equation
(see, e.g., \cite{Faddeev}). However, the numerical realization of this method fails
when the spectral parameter is relatively large. Another approach is based on
the eigenfunction expansion of the Green function (see, e.g., \cite[Sect.
5.3]{Hanson}) in which both the discrete and the continuous parts of the
spectrum are represented. The computation of the generalized eigenfunctions
(corresponding to the continuous spectrum) is in general a complicated task.
In some cases the contribution of the discrete spectrum is considered
asymptotically dominant and consequently the contribution of the continuous
spectrum is neglected (see, e.g., \cite{Barrera1} and \cite{Barrera2}). In
this way an intrinsic error is introduced in the computation.

Here we deal with the problem of computation of the scattering data and
the Green function without neglecting the contribution of the continuous
spectrum eigenfunctions, improving thus the accuracy of the computation. For
this a recent result from \cite{KNT 2015} is used due to which the solutions
of the Schr\"{o}dinger equation
\[
-y^{\prime\prime}+q(x)y=\omega^{2}y
\]
are found in the form of Neumann series of Bessel functions (NSBF) admitting a
uniform error estimate of approximation with respect to
$\operatorname{Re}\omega$. This is especially convenient for computing the
integrals arising in the eigenfunction expansion of the Green function.

The numerical implementation of the method is discussed, and some numerical
illustrations are given.

The paper is organized as follows. In Section \ref{SectPreliminaries} we
recall some known facts about the Schr\"{o}dinger equation on the whole axis.
In Section \ref{SectNSBF} we recall the main results from \cite{KNT 2015} and
introduce a procedure for constructing solutions of the one-dimensional
Schr\"{o}dinger equation via the NSBF representation. In Section
\ref{SectJost} the representation is used for calculating the Jost solutions
for the Schr\"{o}dinger equation with a compactly supported potential.
In Section \ref{Sect improve} we use the asymptotic expansions of the Jost solutions to improve the convergency of the integrals arising in the eigenfunction expansion of the Green function. In
Section \ref{SectComputations} the numerical implementation of the algorithm
is discussed. Finally, some conclusions are provided.

\section{Preliminaries\label{SectPreliminaries}}

The one-dimensional Schr\"{o}dinger operator
\[
H:=-\frac{d^{2}}{dx^{2}}+q(x)
\]
is considered with $q$ being a real-valued measurable function satisfying the
condition%
\begin{equation}
\int_{-\infty}^{\infty}\left(  1+\left\vert x\right\vert \right)  \left\vert
q(x)\right\vert dx<\infty. \label{Cond q}%
\end{equation}
Under this condition the Schr\"{o}dinger equation%
\begin{equation}
Hy(x):=-y^{\prime\prime}(x)+q(x)y(x)=\omega^{2}y(x),\qquad\omega\in R
\label{Schr}%
\end{equation}
possesses the solutions $y_{1}(x,\omega)$ and $y_{2}(x,\omega)$ with the
following asymptotic behaviour at infinity%
\[
y_{1}(x,\omega)=e^{i\omega x}+o(1),\qquad x\rightarrow\infty
\]
and%
\[
y_{2}(x,\omega)=e^{-i\omega x}+o(1),\qquad x\rightarrow-\infty
\]
(see, e.g., \cite{Faddeev}). The functions $y_{1}(x,\omega)$ and
$y_{2}(x,\omega)$ are frequently referred to as the Jost solutions. They can be analytically extended onto $\operatorname*{Im}\omega\geq0$ with this asymptotics
\cite{Berezin}.

For real $\omega\neq0$ two fundamental sets of solutions can be constructed%
\[
\left\{  y_{1}(x,\omega),\quad y_{1}(x,-\omega)\right\}
\]
and
\[
\left\{  y_{2}(x,\omega),\quad y_{2}(x,-\omega)\right\}  .
\]
Note that
\[
y_{1}(x,-\omega)=\overline{y_{1}(x,\omega)}\qquad\text{and}\qquad y_{2}%
(x,-\omega)=\overline{y_{2}(x,\omega)}.
\]
These two fundamental sets are related by the equalities%
\[
y_{2}(x,\omega)=a(\omega)y_{1}(x,-\omega)+b(\omega)y_{1}(x,\omega),
\]%
\[
y_{2}(x,-\omega)=\overline{b(\omega)}y_{1}(x,-\omega)+\overline{a(\omega
)}y_{1}(x,\omega)
\]
where $a(\omega)$ and $b(\omega)$ are called the transmission and the
reflection coefficients respectively \cite{Berezin}. The coefficients satisfy
the equality \cite{Faddeev}%
\[
\left\vert a(\omega)\right\vert ^{2}=1+\left\vert b(\omega)\right\vert ^{2}.
\]
The transmission coefficient can be expressed in terms of the Wronskian of the
solutions $y_{1}(x,\omega)$ and $y_{2}(x,\omega)$,
\begin{equation}
a(\omega)=-\frac{1}{2i\omega}\left(  y_{1}^{\prime}(x,\omega)y_{2}%
(x,\omega)-y_{1}(x,\omega)y_{2}^{\prime}(x,\omega)\right)  . \label{a(k)}%
\end{equation}
For large $\omega$ the coefficients $a(\omega)$ and $b(\omega)$ satisfy the
following asymptotic relations \cite{Faddeev}, valid for all $\omega$ with
$\operatorname*{Im}\omega\geq0$,%
\begin{equation}\label{a(w) asympt}
a(\omega)=1+\frac{I_{q}}{2i\omega}+o\left(  \frac{1}{\left\vert \omega
\right\vert }\right)  \qquad\text{and}\qquad b(\omega)=o\left(  \frac
{1}{\left\vert \omega\right\vert }\right)
\end{equation}
where $I_{q}:=\int_{-\infty}^{\infty}q(x)dx$.

The transmission coefficient $a(\omega)$ may possess a finite number of simple
purely imaginary zeros $\omega_{1},\ldots\omega_{m}$ whose squares are the
eigenvalues of the operator $H$. Denote by $v_{1}(x),\ldots v_{m}(x)$ the
corresponding eigenfunctions of the discrete spectrum (their eigenvalues equal
$\omega_{j}^{2}$) with norm $1$ in $L_{2}(\mathbb{R})$. Together with the set
of solutions
\begin{equation}
\left\{  u_{1}(x,\omega):=\frac{y_{1}(x,\omega)}{\sqrt{2\pi}a(\omega)},\quad
u_{2}(x,\omega):=\frac{y_{2}(x,\omega)}{\sqrt{2\pi}a(\omega)}\right\}
\label{u12}%
\end{equation}
they form a complete orthonormal system of generalized eigenfunctions of the
operator $H$ \cite[Theorem 6.2]{Berezin}. The equality
\[
\int_{0}^{\infty}\left(  \overline{u_{1}(x,\omega)}u_{1}(y,\omega
)+\overline{u_{2}(x,\omega)}u_{2}(y,\omega)\right)  d\omega+\sum_{j=1}%
^{m}\overline{v_{j}(x)}v_{j}(y)=\delta(x-y)
\]
understood in a distributional sense is valid.

The Green function of the operator $H-\lambda$ admits the representation (see,
e.g., \cite[Sect. 5.3]{Hanson})
\begin{equation}
G(x,y;\lambda)=\sum_{j=1}^{m}\frac{v_{j}(x)\overline{v_{j}(y)}}{\lambda
_{j}-\lambda}+\int_{0}^{\infty}\frac{u_{1}(x,\omega)\overline{u_{1}(y,\omega
)}}{\omega^{2}-\lambda}d\omega+\int_{0}^{\infty}\frac{u_{2}(x,\omega
)\overline{u_{2}(y,\omega)}}{\omega^{2}-\lambda}d\omega, \label{Grepres}%
\end{equation}
$\lambda \not\in  [0,\infty)\cup \{ \lambda_j\}_{j=1}^m$.
In practice, the computation of the Green function is performed more
frequently by means of the formula (see, e.g., \cite{Faddeev})%
\begin{align}
G(x,y;\lambda)&=\frac{y_{1}(x,\sqrt{\lambda})y_{2}(y,\sqrt{\lambda})}%
{2i\sqrt{\lambda}a\left(  \sqrt{\lambda}\right)  },\qquad y<x,
\label{G via Two}\\
G(x,y;\lambda)&=G(y,x;\lambda) \label{G via Two 2}%
\end{align}
where $\lambda\in\mathbb{C}$ and the branch for $\sqrt{\lambda}$ is chosen so
that $\operatorname*{Im}\sqrt{\lambda}\geq0$. If $\lambda\neq\omega_{j}^{2}$
and $\lambda\neq0$ the following estimate holds%
\[
\left\vert G(x,y;\lambda)\right\vert \leq Ce^{-\operatorname*{Im}\sqrt
{\lambda}\left\vert x-y\right\vert }.
\]
\qquad\qquad

Often the Green function must be computed for $\lambda$ being large negative
numbers (see, e.g., \cite{Barrera1}, \cite{Barrera2}). In this case formula
(\ref{G via Two}) presents serious numerical difficulties due to the fact that
one of the solutions is exponentially increasing while the other is
exponentially decreasing, in both cases with $\sqrt{\lambda}$ participating in
the exponential order. In spite of a seemingly more complicated appearance,
formula (\ref{Grepres}) becomes more convenient for computation. However,
usually the integrals appearing in it are not computed due to the obvious
difficulty in calculating the solutions (\ref{u12}) for a large interval with
respect to $\omega$, and $G(x,y;\lambda)$ is approximated by the finite sum
$\sum_{j=1}^{m}\frac{v_{j}(x)\overline{v_{j}(y)}}{\lambda_{j}-\lambda}$ only
(see, e.g., \cite{Barrera1}, \cite{Barrera2}). Occasionally, in some practical
models of wave propagation, such approximation can be justified with the aid
of certain asymptotic considerations. Nevertheless, the interest in accurate
computing of the Green function remains, and in the present work we show that
formula (\ref{Grepres}) combined with the NSBF representation of solutions of
(\ref{Schr}) obtained in \cite{KNT 2015} offers a viable\ way to perform such computations.

\section{NSBF representation for solutions of the one-dimensional
Schr\"{o}dinger equation\label{SectNSBF}}

Throughout the paper we suppose that $f$ is a non-vanishing solution (in
general, complex-valued) of the equation
\begin{equation}
Hf=0 \label{SLhom}%
\end{equation}
satisfying the initial condition
\[
f(0)=1.
\]
Since $q$ is real valued, such $f$ can be chosen, e.g., as the following
combination $f_{0}+if_{1}$ of two solutions of (\ref{SLhom}) satisfying the
initial conditions
\begin{equation}
f_{0}(0)=1,\qquad f_{0}^{\prime}(0)=0 \label{f0}%
\end{equation}
and
\[
f_{1}(0)=0,\qquad f_{1}^{\prime}(0)=1.
\]
Denote $h:=f^{\prime}(0)$.

Consider two sequences of recursive integrals (see \cite{KrCV08},
\cite{KrPorter2010})
\[
X^{(0)}(x)\equiv1,\qquad X^{(n)}(x)=n\int_{0}^{x}X^{(n-1)}(s)\left(
f^{2}(s)\right)  ^{(-1)^{n}}\,\mathrm{d}s,\qquad n=1,2,\ldots
\]
and
\[
\widetilde{X}^{(0)}\equiv1,\qquad\widetilde{X}^{(n)}(x)=n\int_{0}%
^{x}\widetilde{X}^{(n-1)}(s)\left(  f^{2}(s)\right)  ^{(-1)^{n-1}}%
\,\mathrm{d}s,\qquad n=1,2,\ldots.
\]

\begin{definition}
\label{Def Formal powers phik and psik}The families of functions $\left\{
\varphi_{k}\right\}  _{k=0}^{\infty}$ and $\left\{  \psi_{k}\right\}
_{k=0}^{\infty}$ constructed according to the rules
\begin{equation}
\varphi_{k}(x)=%
\begin{cases}
f(x)X^{(k)}(x), & k\text{\ odd},\\
f(x)\widetilde{X}^{(k)}(x), & k\text{\ even}%
\end{cases}
\label{phik}%
\end{equation}
and
\begin{equation}
\psi_{k}(x)=%
\begin{cases}
\dfrac{\widetilde{X}^{(k)}(x)}{f(x)}, & k\text{\ odd,}\\
\dfrac{X^{(k)}(x)}{f(x)}, & k\text{\ even}.
\end{cases}
\label{psik}%
\end{equation}
are called the systems of formal powers associated with $f$.
\end{definition}

\begin{remark}
The formal powers arise in the spectral parameter power series (SPPS)
representation for solutions of \eqref{Schr} (see \cite{KKRosu},
\cite{KrCV08}, \cite{KMoT}, \cite{KrPorter2010}).
\end{remark}

Let $c(x,\omega)$, $s(x,\omega)$ denote the solutions of (\ref{Schr})
satisfying the initial conditions in the origin%
\begin{equation}
c(0,\omega)=1,\qquad c^{\prime}(0,\omega)=h,\qquad s(0,\omega)=0,\qquad
s^{\prime}(0,\omega)=\omega, \label{cs init}%
\end{equation}
where $h=f^{\prime}(0)\in\mathbb{C}$.

\begin{theorem}[\cite{KNT 2015}]
\label{Th Representation of solutions via Bessel} The solutions
$c(x,\omega)$ and $s(x,\omega)$ of \eqref{Schr} admit the following
representations%
\begin{equation}
c(x,\omega)=\cos\omega x+2\sum_{n=0}^{\infty}(-1)^{n}\beta_{2n}(x)j_{2n}%
(\omega x) \label{c}%
\end{equation}
and%
\begin{equation}
s(x,\omega)=\sin\omega x+2\sum_{n=0}^{\infty}(-1)^{n}\beta_{2n+1}%
(x)j_{2n+1}(\omega x), \label{s}%
\end{equation}
where $j_{m}$ stands for the spherical Bessel function of order $m$, the
functions $\beta_{n}$ are defined as follows
\begin{equation}
\beta_{n}(x)=\frac{2n+1}{2}\biggl(\sum_{k=0}^{n}\frac{l_{k,n}\varphi_{k}%
(x)}{x^{k}}-1\biggr), \label{beta direct definition}%
\end{equation}
where $l_{k,n}$ is the coefficient of $x^{k}$ in the Legendre polynomial of
order $n$. The series in \eqref{c} and \eqref{s} converge uniformly with
respect to $x$ on any segment and converge uniformly with respect to $\omega$
on any compact subset of the complex plane of the variable $\omega$. Moreover,
for the functions
\begin{equation}
c_{N}(x,\omega)=\cos\omega x+2\sum_{n=0}^{N}(-1)^{n}\beta_{2n}(x)j_{2n}(\omega
x) \label{cN}%
\end{equation}
and
\begin{equation}
s_{N}(x,\omega)=\sin\omega x+2\sum_{n=0}^{N}(-1)^{n}\beta_{2n+1}%
(x)j_{2n+1}(\omega x) \label{sN}%
\end{equation}
the following estimates hold%
\begin{equation}
\left\vert c(x,\omega)-c_{N}(x,\omega)\right\vert \leq2|x|\varepsilon
_{N}(x)\qquad\text{and}\qquad\left\vert s(x,\omega)-s_{N}(x,\omega)\right\vert
\leq2|x|\varepsilon_{N}(x) \label{estc1}%
\end{equation}
for any $\omega\in\mathbb{R}$, $\omega\neq0$, and
\begin{equation}
\left\vert c(x,\omega)-c_{N}(x,\omega)\right\vert \leq\frac{2\varepsilon
_{N}(x)\,\sinh(Cx)}{C}\qquad\text{and}\qquad\left\vert s(x,\omega
)-s_{N}(x,\omega)\right\vert \leq\frac{2\varepsilon_{N}(x)\,\sinh(Cx)}{C}
\label{estc2}%
\end{equation}
for any $\omega\in\mathbb{C}$, $\omega\neq0$ belonging to the strip
$\left\vert \operatorname{Im}\omega\right\vert \leq C$, $C\geq0$, where
$\varepsilon_N$ is a nonnegative function decreasing to zero as $N\to\infty$ (an estimate for $\varepsilon_{N}(x) $
is presented in \cite{KNT 2015}).
\end{theorem}

\begin{theorem}[\cite{KNT 2015}]
\label{Th Derivatives Schrod} The derivatives of the solutions
$c(x,\omega)$ and $s(x,\omega)$ with respect to $x$ admit the following
representations
\begin{equation}
c^{\prime}(x,\omega)=-\omega\sin\omega x+\left(  h+\frac{1}{2}\int_{0}%
^{x}q(s)\,ds\right)  \cos\omega x+2\sum_{n=0}^{\infty}(-1)^{n}\gamma
_{2n}(x)j_{2n}(\omega x) \label{c prime (omega)}%
\end{equation}
and
\begin{equation}
s^{\prime}(x,\omega)=\omega\cos\omega x+\frac{1}{2}\left(  \int_{0}%
^{x}q(s)\,ds\right)  \sin\omega x+2\sum_{n=0}^{\infty}(-1)^{n}\gamma
_{2n+1}(x)j_{2n+1}(\omega x) \label{s prime (omega)}%
\end{equation}
where $\gamma_{n}$ are defined as follows
\begin{equation}
\gamma_{n}(x)=\frac{2n+1}{2}\left(  \sum_{k=0}^{n}\frac{l_{k,n}}{x^k}\left(
k\psi_{k-1}(x)+\frac{f^{\prime}(x)\varphi_{k}(x)}{f(x)}\right)  -\frac
{n(n+1)}{2x}-\frac{1}{2}\int_{0}^{x}q(s)\,ds-\frac{h}{2}\left(  1+(-1)^{n}%
\right)  \right)  . \label{gamma n}%
\end{equation}
The series in \eqref{c prime (omega)} and \eqref{s prime (omega)} converge
uniformly with respect to $x$ on any segment and converge uniformly with
respect to $\omega$ on any compact subset of the complex plane of the variable
$\omega$.

Moreover, for the approximations
\begin{equation}
\overset{\circ}{c}_{N}(x,\omega):=-\omega\sin\omega x+\left(  h+\frac{1}%
{2}\int_{0}^{x}q(s)\,ds\right)  \cos\omega x+2\sum_{n=0}^{N}(-1)^{n}%
\gamma_{2n}(x)j_{2n}(\omega x) \label{cN prime}%
\end{equation}
and
\begin{equation}
\overset{\circ}{s}_{N}(x,\omega):=\omega\cos\omega x+\frac{1}{2}\left(
\int_{0}^{x}q(s)\,ds\right)  \sin\omega x+2\sum_{n=0}^{N}(-1)^{n}\gamma
_{2n+1}(x)j_{2n+1}(\omega x) \label{sN prime}%
\end{equation}
the following inequalities are valid%
\[
\left\vert c^{\prime}(x,\omega)-\overset{\circ}{c}_{N}(x,\omega)\right\vert
\leq2|x|\varepsilon_{1,N}(x)\qquad\text{and}\qquad\left\vert s^{\prime
}(x,\omega)-\overset{\circ}{s}_{N}(x,\omega)\right\vert \leq2|x|\varepsilon
_{1,N}(x)
\]
for any $\omega\in\mathbb{R}$, $\omega\neq0$, and%
\[
\left\vert c^{\prime}(x,\omega)-\overset{\circ}{c}_{N}(x,\omega)\right\vert
\leq\frac{2\varepsilon_{1,N}(x)\,\sinh(Cx)}{C}\qquad\text{and}\qquad\left\vert
s^{\prime}(x,\omega)-\overset{\circ}{s}_{N}(x,\omega)\right\vert \leq
\frac{2\varepsilon_{1,N}(x)\,\sinh(Cx)}{C}%
\]
for any $\omega\in\mathbb{C}$, $\omega\neq0$ belonging to the strip
$\left\vert \operatorname{Im}\omega\right\vert \leq C$, $C\geq0$, where
$\varepsilon_{1,N}$ is a nonnegative function decreasing to zero as $N\to\infty$.
\end{theorem}

\begin{remark}
The inequalities \eqref{estc1} and \eqref{estc2} are of particular importance
when using representations \eqref{c} and \eqref{s} because they guarantee a
uniform ($\omega$-independent) error estimate for an approximation of solutions which was
illustrated by numerical experiments in \cite{KNT 2015}.
\end{remark}

\begin{remark}\label{ComputationBetaGamma}
For numerical implementation of the NSBF representations the formulas \eqref{beta direct definition} and \eqref{gamma n} are not the best options. Due to the fast growth of the Legendre polynomial's coefficients only few functions $\beta_n$ and $\gamma_n$ can be computed with sufficient accuracy. Better option consists in using the alternative recurrent formulas proposed in \cite[Section 6]{KNT 2015}.
\end{remark}

\section{Construction of the Jost solutions\label{SectJost}}

The condition (\ref{Cond q}) indicates a sufficiently fast decay of
$\left\vert q(x)\right\vert $ when $\left\vert x\right\vert \rightarrow\infty
$. This means that a compactly supported potential can be a reasonable
approximation of $q$. Let us consider the operator $H$ with a compactly
supported potential $q$, and without loss of generality suppose that
$\operatorname*{supp}q\subset\left[  0,d\right]  $, $d>0$. Then the solution
$y_{2}(x,\omega)$ satisfies the initial conditions
\begin{equation}\label{y2IC}
y_{2}(0,\omega)=1\qquad\text{and}\qquad y_{2}^{\prime}(0,\omega)=-i\omega,
\end{equation}
while the initial conditions for $y_{1}(x,\omega)$ at the point $d$ have the
form
\begin{equation}\label{y1IC}
y_{1}(d,\omega)=e^{i\omega d}\qquad\text{and}\qquad y_{1}^{\prime}%
(d,\omega)=i\omega e^{i\omega d}.
\end{equation}
We are looking for $y_{1}(x,\omega)$ and $y_{2}(x,\omega)$ in the form of
linear combinations of $c(x,\omega)$ and $s(x,\omega)$. Since
\[
c(0,\omega)=1,\qquad c^{\prime}(0,\omega)=h,\qquad s(0,\omega)=0,\qquad
s^{\prime}(0,\omega)=\omega
\]
we obtain
\begin{equation}
y_{1}(x,\omega)=\frac{e^{i\omega d}}{\omega}\Bigl(  \bigl(  s^{\prime
}(d,\omega)-i\omega s(d,\omega)\bigr)  c(x,\omega)+\bigl(  i\omega
c(d,\omega)-c^{\prime}(d,\omega)\bigr)  s(x,\omega)\Bigr)  \label{y1}%
\end{equation}
and%
\begin{equation}
y_{2}(x,\omega)=c(x,\omega)-is(x,\omega)-\frac{h}{\omega}s(x,\omega).
\label{y2}%
\end{equation}
Thus, we are in a position to use the NSBF representation for computing
$a(\omega)$ and the orthonormalized eigenfunctions of the continuous spectrum
(\ref{u12}) as well as the orthonormalized eigenfunctions of the discrete
spectrum $v_{1}(x),\ldots v_{m}(x)$ (which can also be calculated using the
SPPS method described in \cite{CastilloKrav}). Taking into account that
$a(\omega)$ can be singular in the origin it is convenient to calculate the
values of the functions (\ref{u12}) for $\omega=0$. This result we formulate
in the following statement.

\begin{theorem}
Let $q$ be real valued, $\operatorname*{supp}q\subset\left[  0,d\right]  $,
$d>0$ and $q\in C\left[  0,d\right]  $. Then the following asymptotic relation
is valid
\begin{equation}
a(\omega)=\frac{f_{0}^{\prime}(d)}{2i\omega}-\frac{1}{2}\left(  f_{0}%
(d)-df_{0}^{\prime}(d)+\varphi_{1}^{\prime}(d)\right)  +O(\omega
)\qquad\text{when }\omega\rightarrow0 \label{a(k) asympt f0}%
\end{equation}
where $f_{0}$ satisfies \eqref{SLhom} and \eqref{f0}.

The transmission coefficient $a(\omega)$ has a simple pole in the origin if
and only if zero is not a Neumann eigenvalue of the operator $H$ on $\left[
0,d\right]  $, and in this case%
\begin{equation}
u_{1}(x,0)=u_{2}(x,0)=0. \label{u1u2 origin}%
\end{equation}

Otherwise,
\begin{equation}
a(\omega)=-\frac{1}{2}\left(  f_{0}(d)+\varphi_{1}^{\prime}(d)\right)
+O(\omega)\qquad\text{when }\omega\rightarrow0, \label{a(k) asympt f0 special}%
\end{equation}
$f_{0}(d)+\varphi_{1}^{\prime}(d)\neq0$, and
\begin{equation}
u_{1}(x,0)=-\sqrt{\frac{2}{\pi}}\frac{f_{0}(x)}{f_{0}^{2}(d)+1}\qquad
\text{and}\qquad u_{2}(x,0)=-\sqrt{\frac{2}{\pi}}\frac{f_{0}(x)f_{0}(d)}%
{f_{0}^{2}(d)+1}. \label{u1u2 origin special}%
\end{equation}

\end{theorem}

\begin{proof}
From the NSBF representations (\ref{c}), (\ref{s}) and (\ref{c prime (omega)}), (\ref{s prime (omega)}) of the solutions
$c(x,\omega)$ and $s(x,\omega)$  and their derivatives as well as from formulas from \cite{KNT 2015} we find that
\begin{align*}
c(x,\omega)&=f(x)+O(\omega^{2}),\qquad s(x,\omega)=\varphi_{1}(x)\omega
+O(\omega^{3}),\\
c'(x,\omega)& = f'(x)+O(\omega^2),\qquad s'(x,\omega) = \varphi_1'(x)\omega + O(\omega^3),
\qquad\text{when }\omega\rightarrow0
\end{align*}
and hence
\begin{align*}
y_{1}(x,\omega)  &  =f(x)\varphi_{1}^{\prime}(d)-f^{\prime}(d)\varphi_{1}(x)\\
&\quad   +i\omega\Bigl(  f(d)\varphi_{1}(x)-f(x)\varphi_{1}(d)+d\left(
f(x)\varphi_{1}^{\prime}(d)-f^{\prime}(d)\varphi_{1}(x)\right)  \Bigr)
+O(\omega^{2}),\\
y_{2}(x,\omega)&=f(x)-h\varphi_{1}(x)-i\varphi_{1}(x)\omega+O(\omega^{2}%
),\qquad\text{when }\omega\rightarrow0.
\end{align*}%
Thus, from (\ref{a(k)}) we obtain%
\begin{equation}
a(\omega)=\frac{f^{\prime}(d)-h\varphi_{1}^{\prime}(d)}{2i\omega}-\frac{1}%
{2}\Bigl(  f(d)-h\varphi_{1}(d)-d\left(  f^{\prime}(d)-h\varphi_{1}^{\prime
}(d)\right)  +\varphi_{1}^{\prime}(d)\Bigr)  +O(\omega)\quad\text{when
}\omega\rightarrow0. \label{a(k) asympt}%
\end{equation}
Notice that the magnitude $f(d)-h\varphi_{1}(d)$ is invariant with respect to
the choice of the particular solution $f(x)$ satisfying the condition
$f(0)=1$, see \cite[Proposition 4.7]{KT 2013}. In particular, for $f_{0}$ (\ref{a(k) asympt}) takes the form
(\ref{a(k) asympt f0}), from which we see that\ $a(\omega)$ has a simple pole
in the origin iff $f_{0}^{\prime}(d)\neq0$, that is, iff zero is not a Neumann
eigenvalue of the operator $H$ on $\left[  0,d\right]  $. In this case we
obtain (\ref{u1u2 origin}).

Consider the opposite situation when $f_{0}^{\prime}(d)=0$. It is easy to see
that in this case $f_{0}(d)+\varphi_{1}^{\prime}(d)\neq0$. Indeed, when
$f_{0}^{\prime}(d)=0$, we have that $\varphi_{1}^{\prime}(d)=1/f_{0}(d)$
($f_{0}(d)\neq0$ due to the uniqueness of the solution of the Cauchy problem).
The expression $f_{0}(d)+1/f_{0}(d)\neq0$ because otherwise $f_{0}(d)=\pm i$
which is impossible due to the real-valuedness of $q$. Thus, iff zero is a
Neumann eigenvalue of the operator $H$ on $\left[  0,d\right]  $ the
asymptotic behaviour of the transmission coefficient $a(\omega)$ when
$\omega\rightarrow0$ is described by (\ref{a(k) asympt f0 special}) where the
first term of the asymptotics is different from zero. Consequently, in this
case the relations (\ref{u1u2 origin special}) hold.
\end{proof}

\begin{example}
\label{Example}Consider the potential
\[
q(x)=
\begin{cases}
C, & x\in\left[  0,d\right] \\
0, & x\notin\left[  0,d\right]
\end{cases}
\]
where $C$ is a real constant. Take $h=0$. A simple calculation leads to the
following relations
\[
c(x,\omega)=\cos\alpha x,\qquad s(x,\omega)=\frac{\omega}{\alpha}\sin\alpha
x,
\]
where $\alpha:=\sqrt{\omega^{2}-C}$, and
\begin{align*}
y_{1}(x,\omega)&=e^{i\omega d}\left(  \cos\alpha\left(  d-x\right)
-\frac{i\omega}{\alpha}\sin\alpha\left(  d-x\right)  \right)  ,\\
y_{1}^{\prime}(x,\omega)&=e^{i\omega d}\left(  \alpha\sin\alpha\left(
d-x\right)  +i\omega\cos\alpha\left(  d-x\right)  \right)  ,\\
y_{2}(x,\omega)&=\cos\alpha x-\frac{i\omega}{\alpha}\sin\alpha x,\\
y_{2}^{\prime}(x,\omega)&=-\alpha\sin\alpha x-i\omega\cos\alpha x.
\end{align*}
Calculation of $a(\omega)$ according to \eqref{a(k)} gives us the following
expression%
\begin{equation}
a(\omega)=-e^{i\omega d}\left(  \cos\alpha d+\frac{\omega^{2}+\alpha^{2}%
}{2i\omega\alpha}\sin\alpha d\right)  . \label{a(k) Example}%
\end{equation}
From here as well as by means of \eqref{a(k) asympt f0} it can be seen that
$a(\omega)$ has a simple pole in the origin except when $\sin\left(  i\sqrt
{C}d\right)  =0$. This is possible if only $C<0$ and $d$ is such that
$\sqrt{-C}d=n\pi$, $n\in\mathbb{Z}$.

Note that with the aid of \eqref{G via Two} and \eqref{G via Two 2} the Green
function can be written in the form%
\begin{align*}
G(x,y;\omega^{2})&=\frac{e^{i\omega d}}{2i\omega a\left(  \omega\right)
}\left(  \cos\alpha\left(  d-x+y\right)  -\frac{i\omega}{\alpha}\sin
\alpha\left(  d-x+y\right)  \right.\\
 &\quad \left.-\frac{C}{\alpha^{2}}\sin\alpha\left(  d-x\right)  \sin\alpha
y\right)  ,\qquad y<x,
\end{align*}
which together with \eqref{G via Two 2} defines it for all values of $x$ and
$y$. Here $a\left(  \omega\right)  $ is defined by \eqref{a(k) Example}.
\end{example}

\section{Improving the convergence of the integrals in \eqref{Grepres}}\label{Sect improve}
The solutions $c(x,\omega)$ and $s(x,\omega)$ satisfying the initial conditions \eqref{cs init} possess the following asymptotics for real values of $\omega$ (see, e.g., \cite[Section 1.1]{Yurko}, see also formulas (5.2) and (5.3) from \cite{KNT 2015})
\begin{align}
c(x,\omega) &= \cos\omega x+ \left(h+\frac {Q(x)}2\right)\frac{\sin\omega x}{\omega}+o(1/\omega),\label{c(x,w) asympt}\\
c'(x,\omega) &= -\omega\sin\omega x+\left(h+\frac {Q(x)}2\right)\cos\omega x+o(1),\label{cp(x,w) asympt}\\
\displaybreak[2]
s(x,\omega) &= \sin\omega x- \frac {Q(x)}2\frac{\cos\omega x}{\omega} + o(1/\omega),\label{s(x,w) asympt}\\
s'(x,\omega) &= \omega\cos\omega x + \frac {Q(x)}2\sin\omega x + o(1),\qquad \omega\to\pm\infty,
\label{sp(x,w) asympt}
\end{align}
where
\[
Q(x):= \int_0^x q(t)\,dt.
\]
For absolutely continuous on $[0,d]$ potentials the error terms $o(1/\omega)$ and $o(1)$ can be changed to $O(1/\omega^2)$ and $O(1/\omega)$, respectively.
Using this asymptotics  as well as \eqref{a(w) asympt} and formulas \eqref{y1} and \eqref{y2} one can verify that the integrands from \eqref{Grepres} possess the following asymptotics
\begin{equation}\label{integrand asympt}
    \begin{split}
    u_1(x,\omega)\overline{u_1(y,\omega)}+u_2(x,\omega)\overline{u_2(y,\omega)} & = \frac{y_1(x,\omega)\overline{y_1(y,\omega)}+y_2(x,\omega)\overline{y_2(y,\omega)}}{2\pi |a(\omega)|^2}\\
    & = \frac{2\cos\omega(x-y) + \frac{(Q(x)-Q(y))\sin\omega(x-y)}\omega+o(1/\omega)}{2\pi |a(\omega)|^2} \\
    & = \frac { \cos \omega(x-y)}\pi+\frac{(Q(x)-Q(y))\sin\omega(x-y)}{2\pi\omega}+
    o\left(\frac 1\omega\right),\\
    &\quad\quad \omega\to+\infty.
    \end{split}
\end{equation}
Again, for absolutely continuous on $[0,d]$ potentials the error term can be changed to $O(1/\omega^2)$.

The asymptotic relation \eqref{integrand asympt} explains the slow convergence of the integrals in \eqref{Grepres}, the integrands decay as $1/\omega^2$ as $\omega\to+\infty$. At the same time, it provides the way to improve the convergence rate. Note that (c.f., formulas 2.5.9.5, 2.5.9.10 and 2.5.9.11 from \cite{Prudnikov})
\begin{equation*}\label{Green unperturbed}
    \int_0^\infty \frac{\cos\omega(x-y)}{\pi(\omega^2-\lambda)}d\omega = -\frac{e^{i|x-y|\sqrt{\lambda}}}{2i\sqrt{\lambda}},
\end{equation*}
where the branch for $\sqrt{\lambda}$ is chosen so that $\operatorname{Im}\sqrt{\lambda} > 0$ (such branch is uniquely defined since the formula \eqref{Grepres} can be used only for $\lambda\not\in [0,\infty)$), and
\begin{equation*}\label{Second term}
    \begin{split}
        \int_0^\infty \frac{\sin\omega|x-y|}{2\pi\omega(\omega^2-\lambda)}d\omega & = \int_0^\infty \frac{(\omega^2-\lambda)\sin\omega|x-y|}{2\pi\omega(\omega^2-\lambda)^2}d\omega\\
        &= \int_0^\infty \frac{\omega\sin\omega|x-y|}{2\pi(\omega^2-\lambda)^2}d\omega - \int_0^\infty \frac{\lambda\sin\omega|x-y|}{2\pi\omega(\omega^2-\lambda)^2}d\omega\\
        &=-\frac{|x-y|e^{i|x-y|\sqrt{\lambda}}}{8i\sqrt{\lambda}} - \frac 1{4\lambda} + \frac{e^{i|x-y|\sqrt{\lambda}}}{4\lambda} + \frac{|x-y|e^{i|x-y|\sqrt{\lambda}}}{8i\sqrt{\lambda}} \\
        &= \frac{e^{i|x-y|\sqrt{\lambda}} - 1}{4\lambda}.
    \end{split}
\end{equation*}
Hence we may rewrite the formula \eqref{Grepres} either as
\begin{equation}\label{Grepres2}
\begin{split}
    G(x,y;\lambda) &= \sum_{j=1}^{m}\frac{v_{j}(x)\overline{v_{j}(y)}}{\lambda
_{j}-\lambda}-\frac{e^{i|x-y|\sqrt{\lambda}}}{2i\sqrt{\lambda}}\\
&\quad +
\int_{0}^{\infty}\frac{u_{1}(x,\omega)\overline{u_{1}(y,\omega
)}+u_{2}(x,\omega)\overline{u_{2}(y,\omega)} - \frac 1\pi\cos\omega(x-y)}{\omega^{2}-\lambda}d\omega
\end{split}
\end{equation}
or as
\begin{equation}\label{Grepres3}
\begin{split}
    G(x,y;\lambda) &= \sum_{j=1}^{m}\frac{v_{j}(x)\overline{v_{j}(y)}}{\lambda
_{j}-\lambda}-\frac{e^{i|x-y|\sqrt{\lambda}}}{2i\sqrt{\lambda}}+\frac{e^{i|x-y|\sqrt{\lambda}} - 1}{4\lambda}\cdot\frac{x-y}{|x-y|}\bigl(Q(x)-Q(y)\bigr)\\
&\quad +
\int_{0}^{\infty}\frac{u_{1}(x,\omega)\overline{u_{1}(y,\omega
)}+u_{2}(x,\omega)\overline{u_{2}(y,\omega)} - \frac{2\omega\cos\omega(x-y)+ (Q(x)-Q(y))\sin\omega(x-y)}{2\pi\omega}}{\omega^{2}-\lambda}d\omega.
\end{split}
\end{equation}
Now the integrand in \eqref{Grepres2} can be estimated by the expression $1/\omega^3$ as $\omega\to+\infty$, while in \eqref{Grepres3} (for absolutely continuous on $[0,d]$ potentials) by the expression $1/\omega^4$  resulting in a faster integral's convergence. We give a numerical illustration in Section \ref{SectComputations}.

\section{Numerical implementation\label{SectComputations}}

The procedure for the computation of the Green function based on the NSBF
representations was implemented using Matlab 2017a.
The first step is to compute a nonvanishing solution $f$ of (\ref{SLhom})
together with its first derivative (see \cite{KrCV08}, \cite{KrPorter2010} for
 details). Then the systems of formal powers (\ref{phik}) and (\ref{psik})
associated to $f$ are computed. Next, a number of the functions $\beta_{k}$,
and $\gamma_{k}$ are computed according to Remark \ref{ComputationBetaGamma} which lead to the approximate NSBF representations
$c_{N}(x,\omega)$ and $s_{N}(x,\omega)$ of solutions (see (\ref{cN}) and
(\ref{sN})) as well as of their derivatives (\ref{cN prime}) and
(\ref{sN prime}). Then the Jost solutions $y_{1}(x,\omega)$ and $y_{2}%
(x,\omega)$ can be computed using (\ref{y1}) and (\ref{y2}), which allow one
to compute $u_{1}(x,\omega)$ and $u_{2}(x,\omega)$ and hence $G(x,y;\lambda)$
with the aid of (\ref{Grepres3}). Since the main difficulty usually consists in computing the integrals in (\ref{Grepres3}), we consider an example void of a discrete spectrum.

Notice that to compute the Green function for any value of $\lambda
=:\omega_{c}^{2}$ not belonging to the spectrum of $H$ the integration in
formula (\ref{Grepres3}) requires the solutions $u_{1}(x,\omega)$ and
$u_{2}(x,\omega)$ to be computed for $\omega\in\lbrack0,\infty)$ only. Here is
where the independence with respect to $\omega\in\mathbb{R}$ of the estimate
of the approximation (see (\ref{estc1})) becomes especially attractive.

Computations were performed for the case considered in Example \ref{Example}
with $C=1$ and $d=1$. On the segment $x\in\lbrack0,d]$ equally spaced $1001$
points have been chosen for the representation of all the functions involved, $30$ formal powers have been computed for the
construction of $f$, and $N$ was set as $12$ (for details see \cite[Remark
7.1]{KNT 2015}).


First, we illustrate the slow convergence of the integrals in (\ref{Grepres}) and compare the improvements achieved by using \eqref{Grepres2} and \eqref{Grepres3}. The numerical integration was performed by truncating the integrals for $\omega = \Omega$, $\Omega\le 10^{5}$ and using the six points Newton-Cotes method of
seventh order and the step of
$0.01$. The Green function was computed for two cases, $x=0.97$, $y = 0.069$ (distant points $x$ and $y$) and
$x=0.97$, $y=0.969$ (close points $x$ and $y$) and $\lambda = -4$.

The convergence of the integrals in (\ref{Grepres}) is slow as can be seen on
Fig.\ \ref{Figure1Convergence}. For distant points $x$ and $y$ the second improvement formula \eqref{Grepres3} is considerably better giving 2 orders of $\omega$ faster decay of the integral truncation error. However for close values of $x$ and $y$ one can use the first improvement formula \eqref{Grepres2} as well, thanks to the smallness of the contribution of the second asymptotic term in \eqref{integrand asympt}. In both cases the error stabilizes on the values of the truncation parameter $\Omega$ of about $10^4$, which we will use for all further computations.

\begin{figure}
[ptb]
\centering
\begin{tabular}{cc}
\includegraphics[
height=2.5in,
width=3in
]{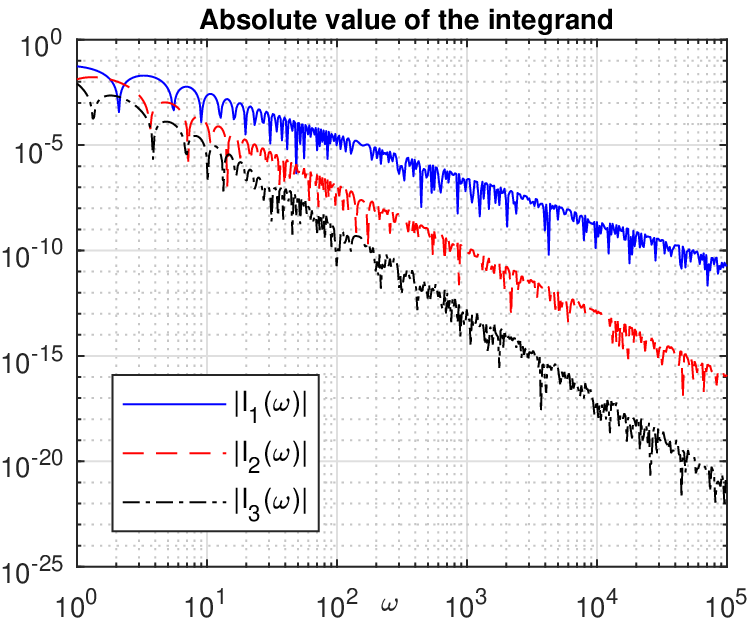} & \includegraphics[
height=2.5in,
width=3in
]{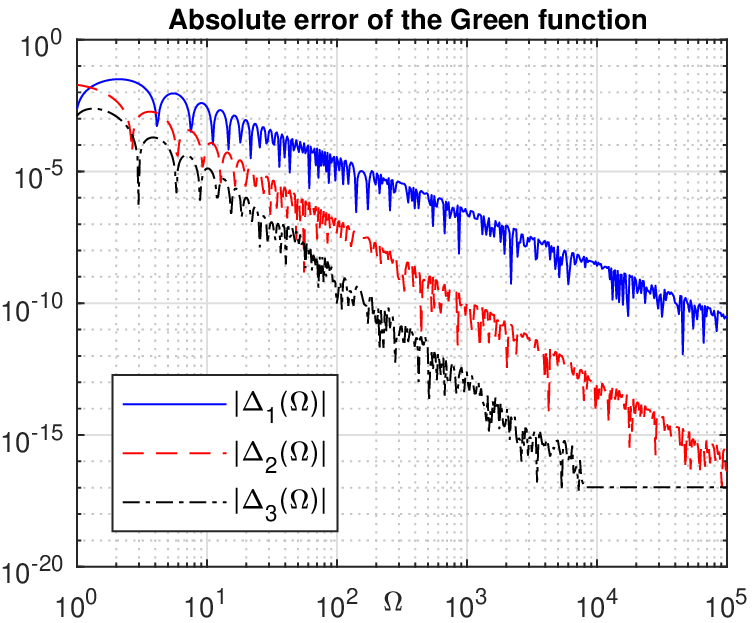}\\
\multicolumn{2}{c}{$x=0.97$, $y=0.069$}\\
\includegraphics[
height=2.5in,
width=3in
]{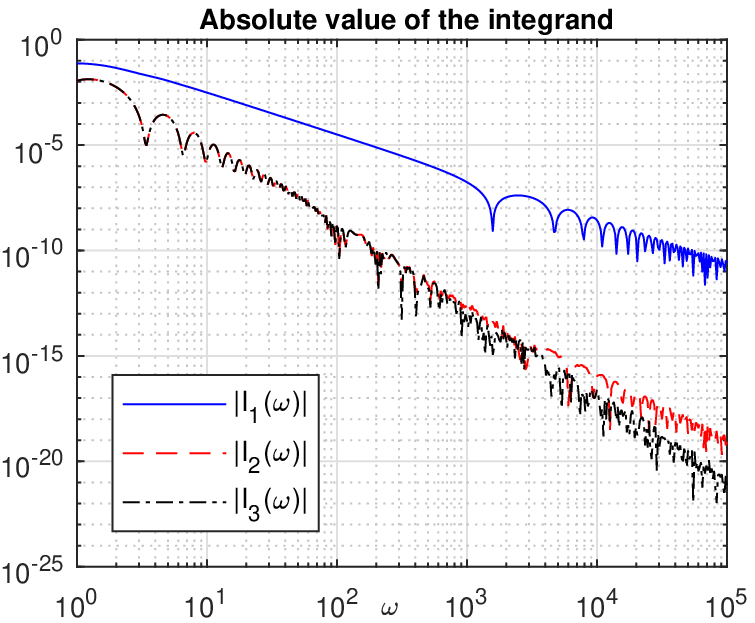}& \includegraphics[
height=2.5in,
width=3in
]{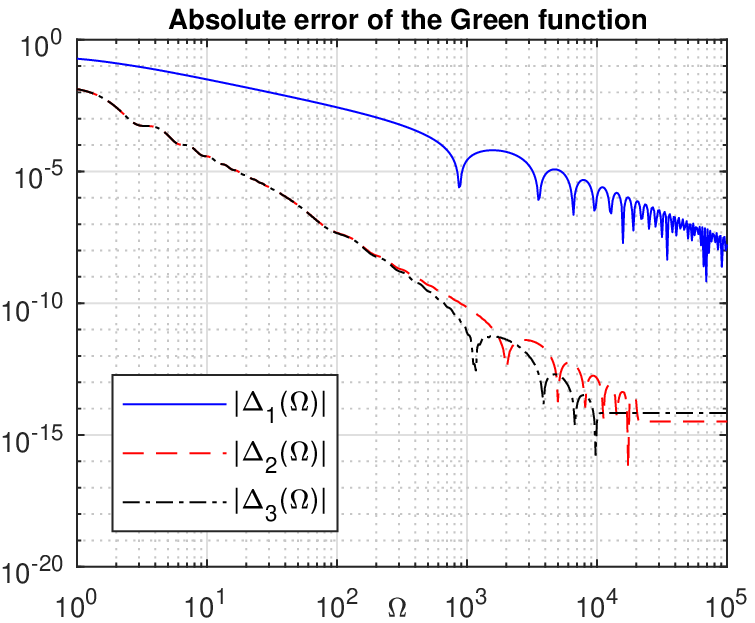}\\
\multicolumn{2}{c}{$x=0.97$, $y=0.969$}
\end{tabular}
\caption{Top left: absolute values of the integrands ($I_1$ -- the sum of the integrands from \eqref{Grepres}, $I_2$ -- the integrand from \eqref{Grepres2} and $I_3$ -- the integrand from \eqref{Grepres3}) for distant points $x=0.97$, $y = 0.069$. Bottom left: absolute values of the same integrands for close points $x=0.97$, $y = 0.969$.
Top right: absolute errors of the computed Green function for distant points $x=0.97$, $y = 0.069$ ($\Delta_1$ -- using formula \eqref{Grepres}, $\Delta_2$ -- using formula \eqref{Grepres2} and $\Delta_3$ -- using formula \eqref{Grepres3}) as the functions of the upper limit $\Omega$ chosen to truncate the integrals. Bottom right: same absolute errors of the computed Green function for close points $x=0.97$, $y = 0.969$.}
\label{Figure1Convergence}
\end{figure}

The absolute error dependence on $x$ for the NSBF method is shown on Figure
\ref{FigAbsErrNSBF} achieving its maximum in the neighborhood of the point
$x=y$.%

\begin{figure}
[ptb]
\centering
\includegraphics[
height=3in,
width=5in,
bb=0 0 360 216
]%
{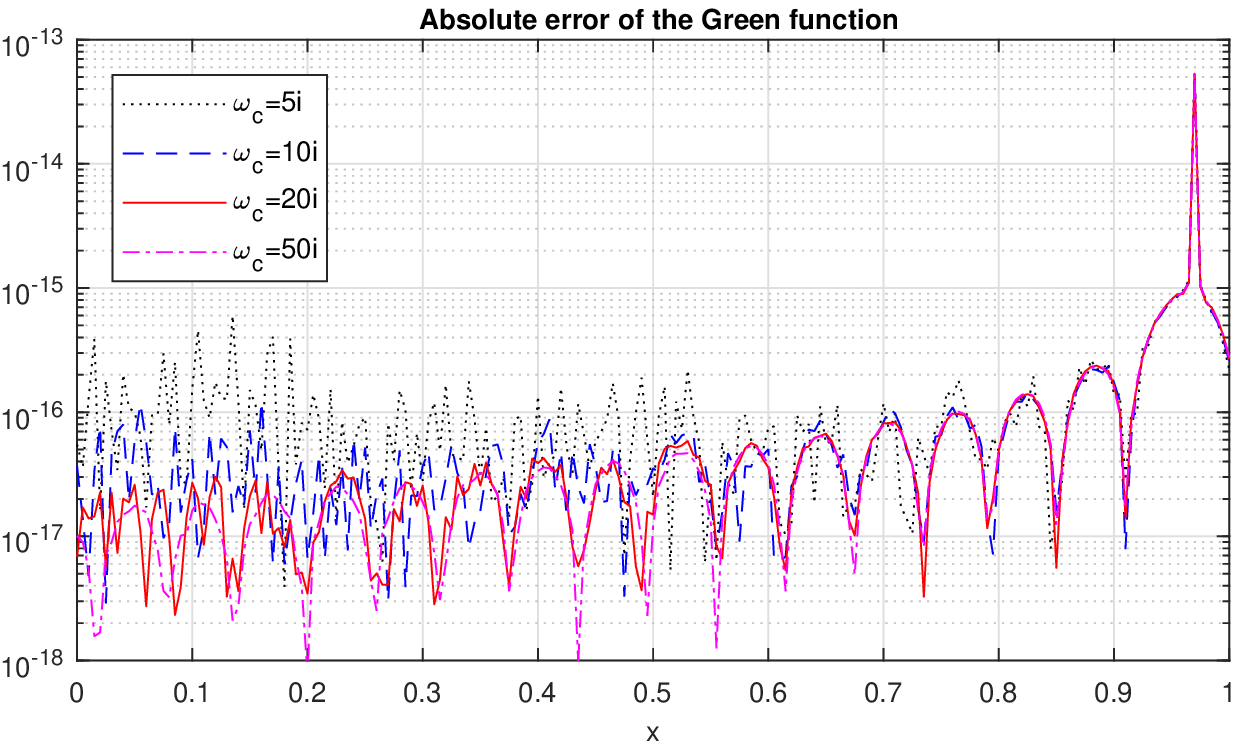}%
\caption{Absolute error of the Green function
$G(x,0.97;\omega_{c}^{2})$ computed using \eqref{Grepres3} for different values of $\omega_{c}$.}
\label{FigAbsErrNSBF}
\end{figure}

Finally, the maximum absolute error for the Green function computed on a domain of the complex plane of the variable $\omega_{c}$ is shown on Figure
\ref{FigAbsErrNSBFComplex}. As can be seen the absolute error of the approximation obtained with the aid of the method proposed here and based on (\ref{Grepres3}) and NSBF representations is uniform with the exception of a small neighborhood of $\omega_c=0$.

For comparison the Green function was also computed by formulas (\ref{G via Two}%
) and (\ref{G via Two 2}) with the aid of Matlab's \texttt{ode45} solver. Note that the computation of the solution $y_1$ by the Matlab solver requires some care. First, since the solution $y_1$ is decreasing, it should be computed ``from right to left'' starting from the initial conditions \eqref{y1IC}. Second, the initial value $e^{i\omega d}$ is very small for the values of $\omega$ having large positive imaginary part. In order to obtain a good accuracy from the \texttt{ode45} solver, one either has to specify a very small (say, $10^{-15}\cdot e^{-\operatorname{Im}\omega d}$) value for the \texttt{`AbsTol'} parameter or to rescale the initial value, setting $\tilde y(d) = 1$, $\tilde y'(d) = i\omega$ and then multiplying the returned function by $e^{i\omega d}$.
Once the computation was performed correctly, the maximum absolute errors of the computed Green function by the proposed method and by the \texttt{ode45} solver were comparable.

\begin{figure}
[ptb]
\centering
\includegraphics[
height=3in,
width=5in,
bb=0 0 360 216
]%
{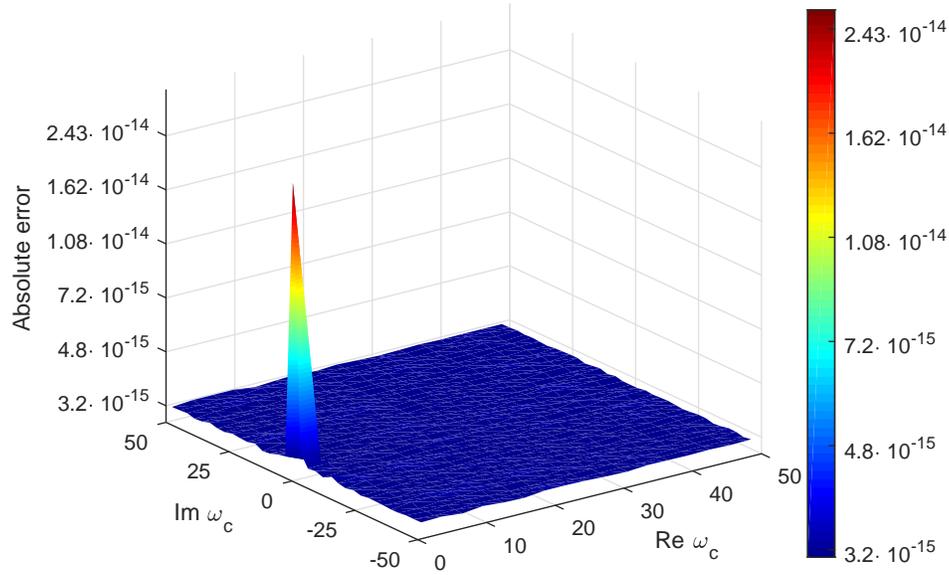}%
\caption{The maximum with respect to $x$ absolute error of the Green function
$G(x,0.97;\omega_{c}^{2})$ computed using \eqref{Grepres3} on a domain of the complex
plane of the variable $\omega_{c}$.}
\label{FigAbsErrNSBFComplex}
\end{figure}

\section*{Conclusions}

A method for computation of the Green function for the one-dimensional
Schr\"{o}dinger equation is proposed based on the eigenfunction expansion of
the Green function and the NSBF representations of the generalized
eigenfunctions. It proved to be a stable and accurate method comparable with the representation in terms of two Jost solutions computed by
the \texttt{ode45} solver provided by Matlab.
%

\section*{Acknowledgements}

R.\ Castillo would like to thank the support of Instituto Polit\'{e}cnico
Nacional through a sabbatical year.
Research of V.\ Kravchenko and S.\ Torba was partially supported by CONACYT,
Mexico via the projects 284470 and 222478.

\end{document}